\documentclass[11pt]{article}
\usepackage{verbatim, latexsym, amssymb, amsmath, amsthm, color, mathabx, tikz-cd}
\usepackage{enumitem}
\usepackage{epsfig}

\usepackage{geometry}
\usepackage{hyperref}

\newtheorem{thm}{Theorem}
\newtheorem{prop}{Proposition}[section]
\newtheorem{lem}[prop]{Lemma}

\theoremstyle{remark}
\newtheorem{rem}{Remark}[section]

\DeclareMathOperator{\bR}{\mathbb{R}}

\DeclareMathOperator{\bZ}{\mathbb{Z}}

\DeclareMathOperator{\cC}{\mathcal{C}}

\DeclareMathOperator{\Diff}{\operatorname{Diff}}
\DeclareMathOperator{\Ham}{\operatorname{Ham}}
\DeclareMathOperator{\Homeo}{\operatorname{Homeo}}

\DeclareMathOperator{\PFC}{\text{PFC}}
\DeclareMathOperator{\PFH}{\text{PFH}}

\DeclareMathOperator{\Hom}{\operatorname{Hom}}

\title{Periodic points of rational area-preserving homeomorphisms}
\author{Rohil Prasad}
\date{}

\begin{document}

\maketitle

\begin{abstract}
    An area-preserving homeomorphism isotopic to the identity is said to have rational rotation direction if its rotation vector is a real multiple of a rational class. We give a short proof that any area-preserving homeomorphism of a compact surface of genus at least two, which is isotopic to the identity and has rational rotation direction, is either the identity or has periodic points of unbounded minimal period. This answers a question of Seyfaddini and can be regarded as a Conley conjecture-type result for symplectic homeomorphisms of surfaces beyond the Hamiltonian case. We also discuss several variations, such as maps preserving arbitrary Borel probability measures with full support, maps not isotopic to the identity, and maps on lower genus surfaces. The proofs of the main results combine topological arguments with periodic Floer homology. 
\end{abstract}

\section{Introduction}

\renewcommand*{\thethm}{\Alph{thm}}

\subsection{History and main results} Questions on the existence and multiplicity of periodic points of area-preserving surface homeomorphisms have a long history, dating back to Poincar\'e and Birkhoff's work on annulus twist maps and the restricted planar three-body problem, and have attracted significant attention since then. The existence question asks whether any periodic points exist, and the multiplicity question asks how many there are. On a closed surface of genus $g \geq 1$, area-preserving maps with any finite number $N \geq 2g - 2$ of periodic points may be constructed by adding $N - 2g + 2$ singularities to an irrational translation flow ($g = 1$) or a translation flow in any minimal direction ($g \geq 2$). The resulting time-one maps have finitely many periodic points, all of which are fixed, and are isotopic to the identity relative to their fixed point set. The current state-of-the-art, to our knowledge, is due to Le Calvez \cite{lecalvez}. He shows, barring some edge cases when $g \leq 1$, that any area-preserving homeomorphism of a closed surface is either periodic, has periodic points of unbounded minimal period, or after passing to an iterate only has finitely many fixed points and is isotopic to the identity relative to the fixed point set. The last case is the one modeled by the example above. 

In this note, we discuss a simple and dense homological condition which forces an area-preserving map isotopic to the identity to have infinitely many periodic points. Fix a compact surface $\Sigma$ and a smooth area form $\omega$. Any area-preserving map $\phi \in \Homeo_0(\Sigma, \omega)$ isotopic to the identity has a \emph{rotation vector} 
$$\mathcal{F}(\phi) \in H_1(\Sigma; \mathbb{R})\,/\,\Gamma_\omega$$ where $\Gamma_\omega \subseteq H_1(\Sigma; \mathbb{Z})$ is a discrete subgroup. The map $\phi$ is said to have \emph{rational rotation direction} if $\mathcal{F}(\phi)$ is a real multiple of a rational class, i.e. $c \cdot \mathcal{F}(\phi) \in H_1(\Sigma; \mathbb{Q})\,/\,\Gamma_\omega$ for some real number $c > 0$. Any area-preserving map can be perturbed to an area-preserving map with rational rotation direction by a $C^\infty$-small perturbation, so a dense subset of maps have rational rotation direction. We now state our main result. 

\begin{thm}\label{thm:main}
    Fix a compact surface $\Sigma$ of genus $\geq 2$ and a smooth area form $\omega$. Let $\phi \in \Homeo_0(\Sigma, \omega)$ be any area-preserving homeomorphism which is isotopic to the identity and rational. Then $\phi$ is either the identity or has periodic points of unbounded minimal period. 
\end{thm}

\begin{rem}
In the theorem above and all subsequent discussion, we allow compact surfaces to have nonempty boundary, and say they are \emph{closed} when the boundary is empty. The genus of a compact surface is the genus of the closed surface obtained by attaching disks to each boundary component. All surfaces are also assumed to be oriented from now on. 
\end{rem}

\begin{rem}
Similar results are either known or can be shown to hold by combining known results for compact surfaces of genus $0$ and $1$. See Section~\ref{subsec:lower_genus} for a more detailed discussion. 
\end{rem}

\begin{rem}
Franks--Handel \cite{franks_handel} and Le Calvez \cite{lecalvez_ham} previously proved Theorem~\ref{thm:main} under the assumption that $\phi$ is Hamiltonian, which is equivalent to the condition $\mathcal{F}(\phi) = 0$.
\end{rem}

\begin{rem} \label{rem:infinitely_many}
Fix a compact surface $\Sigma$ of genus $\geq 2$ and a smooth area form $\omega$. Theorem~\ref{thm:main} and a short Baire category argument show that a $C^\infty$-generic $\phi \in \Diff_0(\Sigma, \omega)$ has periodic points of unbounded minimal period. Any periodic point can be made nondegenerate by a $C^\infty$-small local Hamiltonian perturbation, which does not change the rotation vector. Since rational maps are $C^\infty$-dense, it follows from Theorem~\ref{thm:main} that for each $d$ there is an open and dense subset $\mathcal{U}_d \subset \Diff_0(\Sigma, \omega)$ such that each $\phi \in \mathcal{U}_d$ has a periodic point of minimal period $\geq d$. Each $\phi$ in the intersection $\mathcal{U} := \cap_{d \geq 1} \mathcal{U}_d$ has periodic points of unbounded minimal period. 
\end{rem}

In addition to the historical backdrop presented above, Theorem~\ref{thm:main} fits into an active stream of research in symplectic dynamics centered around the \emph{Conley conjecture}. The original formulation of the conjecture asserts that any Hamiltonian diffeomorphism of a closed aspherical symplectic manifold has infinitely many periodic points. The Conley conjecture was resolved for surfaces by Franks--Handel \cite{franks_handel} and extended to Hamiltonian homeomorphisms by Le Calvez \cite{lecalvez_ham} before being resolved in full generality by breakthrough work of Hingston \cite{Hingston09} for higher-dimensional tori and Ginzburg \cite{Ginzburg10} for the general case. The search for extensions of the Conley conjecture to Hamiltonian diffeomorphisms/homeomorphisms of more general symplectic manifolds has attracted a great deal of ongoing activity and progress \cite{GG09,GG09b,GG10,GG12,Gurel13,GG14,GG15,GG16,Cineli18,GG19}. However, there has been much less progress in establishing Conley conjecture-type results for non-Hamiltonian symplectic maps (some notable results include \cite{Batoreo15,Batoreo17,Batoreo18}). Moreover, to our knowledge, there is no agreed upon formulation of the Conley conjecture for non-Hamiltonian symplectic maps. Theorem~\ref{thm:main} can be viewed not only as establishing such a ``non-Hamiltonian Conley conjecture'' in dimension $2$, but also as a guidepost towards formulating a ``non-Hamiltonian Conley conjecture'' for higher dimensional symplectic manifolds. We pose the following question. 
\vspace{.1in}

\noindent\textbf{Question:} Let $(M, \Omega)$ be a closed and symplectically aspherical symplectic manifold of any dimension. Then does any symplectic diffeomorphism $\phi \in \Diff_0(M, \Omega)$ such that $\mathcal{F}(\phi) \in H_1(M; \mathbb{Q})\,/\,\Gamma_\Omega$ have infinitely many periodic points?

\vspace{.1in}
The idea that a Conley conjecture-type result may hold for area-preserving homeomorphisms with rational rotation direction was motivated by recent work on the $C^\infty$-closing lemma \cite{closing, closing_eh} for area-preserving surface diffeomorphisms.  Herman famously showed \cite[Chapter $4.5$]{hz_book} that a version of the closing lemma using only Hamiltonian perturbations cannot hold for certain irrational maps (Diophantine torus rotations). This issue was avoided by proving a Hamiltonian $C^\infty$-closing lemma for many rational maps\footnote{Any map with rational asymptotic cycle (see Section~\ref{subsec:intro_non_identity}).} and then observing that such maps form a $C^\infty$-dense subset of $\Diff(\Sigma, \omega)$. It seems reasonable to suspect that, given these marked differences in behavior, the rationality condition has dynamical significance. Further evidence that rationality may force infinitely many periodic points is furnished by the fact that the area-preserving maps with finitely many periodic points discussed above all have irrational rotation vector. We also learned while preparing the first version of this paper that Theorem~\ref{thm:main} was also posed as a question, with the same motivation, by Sobhan Seyfaddini.

Theorem~\ref{thm:main} can also be extended, with slightly weaker conclusions, to maps preserving arbitrary full-support Borel probability measures. Given an isotopy $\Phi$ from the identity to $\phi$ and a $\phi$-invariant Borel probability measure $\mu$, a rotation vector $\mathcal{F}(\Phi, \mu) \in H_1(\Sigma; \bR)$ can be defined. When $\Sigma$ has genus $\geq 2$, the rotation vector does not depend on the isotopy $\Phi$, and we write it as $\mathcal{F}(\phi, \mu)$. 

\begin{thm}\label{thm:extension}
Fix a compact surface $\Sigma$ of genus $\geq 2$ and a Borel probability measure $\mu$ with full support such that $\mu(\partial \Sigma) = 0$. Let $\phi \in \Homeo_0(\Sigma, \mu)$ be any $\mu$-preserving homeomorphism which is isotopic to the identity, such that its rotation vector $\mathcal{F}(\phi, \mu) \in H_1(\Sigma; \bR)$ is a real multiple of a rational class. Then $\phi$ has infinitely many periodic points. Moreover, if $\mu$ has no atoms, then $\phi$ is either the identity or has periodic points of unbounded minimal period. 
\end{thm}

Theorem~\ref{thm:extension} follows from a short argument combining Theorem~\ref{thm:main} and a theorem of Oxtoby--Ulam \cite{oxtoby_ulam}, which was suggested to the author by Patrice Le Calvez. Theorem~\ref{thm:main} follows from Theorem~\ref{thm:non_contractible}, which may be of independent interest, and the results in \cite{lecalvez_ham, lecalvez}. 

\begin{thm} \label{thm:non_contractible}
Fix a closed surface $\Sigma$ of genus $\geq 2$ and a smooth area form $\omega$. Then any rational $\phi \in \Homeo_0(\Sigma, \omega)$ is either Hamiltonian or has a non-contractible periodic point. 
\end{thm}

\begin{rem}
An interesting problem related to Theorem~\ref{thm:non_contractible}, posed by Ginzburg, is to determine whether a $C^\infty$-generic Hamiltonian diffeomorphism of a closed and symplectically aspherical symplectic manifold has a non-contractible periodic point. This was proved for the $2$-torus by Le Calvez--Tal \cite{lct18} and was recently extended to all closed surfaces of positive genus by Le Calvez--Sambarino \cite{lct23}. 
\end{rem}

A recent result from symplectic geometry is central to the proof of Theorem~\ref{thm:non_contractible}. In \cite{closing}, a precise non-vanishing theorem is proved for the periodic Floer homology (PFH) of area-preserving diffeomorphisms of closed surfaces with rational rotation vector $(\mathcal{F}(\phi) \in H_1(\Sigma; \mathbb{Q}))$. PFH is a homology theory for area-preserving surface maps built out of their periodic orbits. The non-vanishing theorem relies on a deep result of Lee--Taubes \cite{lee_taubes} showing PFH is isomorphic to monopole Floer homology. We explain more about PFH and state the non-vanishing theorem at the beginning of Section~\ref{sec:pfh}. 

There are two significant issues to overcome in the proof of Theorem~\ref{thm:non_contractible}. The first issue is that PFH is only well-defined for diffeomorphisms, not homeomorphisms. To get around this, we observe that a quantitative version (Proposition~\ref{prop:non_contractible_technical}) of Theorem~\ref{thm:non_contractible} holds; there exists a non-contractible periodic point with an upper bound on its minimal period depending only on the rotation vector. This allows us to extend Theorem~\ref{thm:non_contractible} to homeomorphisms by an approximation argument. 

The second issue is that the non-vanishing of PFH is only known for maps with rotation vector in $H_1(\Sigma; \mathbb{Q})$ and not all maps with rational rotation direction. We cannot hope for too much when $\mathcal{F}(\phi) \not\in H_1(\Sigma; \mathbb{Q})$, in this case there are many examples (irrational torus rotations, translation flows in minimal directions) where PFH essentially vanishes\footnote{More precisely, it vanishes in nontrivial homological gradings, so it only detects null-homologous periodic orbit sets, in which every periodic orbit could be contractible.}. We deal with this by introducing a blow-up argument to reduce from the case of rational rotation direction to the case where $\mathcal{F}(\phi) \in H_1(\Sigma; \mathbb{Q})$. When $\Sigma$ has genus $\geq 2$, $\phi$ has a fixed point, and we can assume without loss of generality that it is contractible since otherwise we would already be done. We blow up this fixed point and attach a disk of a specified area to the boundary circle, and then extend the isotopy over the new surface. The extended map has a rotation vector in the same direction, but scaled by a constant in $(0, 1)$, which depends on the area of the attached disk. We are free to attach a disk of any area, so we scale the rotation vector down to a rational vector, and note that the only new periodic points introduced by blowing up are contractible, so the original map must have a non-contractible periodic point. This argument reveals more generally that, in genus $\geq 2$, the dynamics are essentially identical for maps with $\mathcal{F}(\phi) \in H_1(\Sigma;\mathbb{Q})$ and those with rational rotation direction. 

\subsection{Maps not isotopic to the identity} \label{subsec:intro_non_identity} 

Assume that $\Sigma$ is closed and has genus $\geq 2$. It is known \cite{lecalvez} that either $\phi$ has periodic points of unbounded minimal period or it has periodic Nielsen--Thurston class. Therefore, if $\phi$ has periodic Nielsen--Thurston class and some iterate $\phi^q \in \Homeo_0(\Sigma, \omega)$ has rational rotation direction, Theorem~\ref{thm:main} implies $\phi$ is either periodic or has periodic points of unbounded minimal period. We will now show that maps with periodic Nielsen--Thurston class and rational asymptotic cycle have an iterate with rational rotation direction. 

Let $M_\phi$ denote the mapping torus of $\phi$. The \emph{asymptotic cycle} $\cC(\phi) \in H_1(M_\phi; \bR)$ is an analogue of the rotation vector for area-preserving maps not isotopic to the identity; by rational asymptotic cycle we mean $\cC(\phi) \in H_1(M_\phi; \mathbb{Q})$. It is defined by applying Schwartzman's construction \cite{schwartzman} to the suspension flow. When $\phi$ is isotopic to the identity, a choice of identity isotopy $\Phi$ defines a diffeomorphism $M_\phi \simeq \mathbb{T} \times \Sigma$ where $\mathbb{T} := \mathbb{R}/\mathbb{Z}$ denotes the circle. Lemma~\ref{lem:asymptotic_cycle_computation} proves that, under this identification, $$\cC(\phi) = [\mathbb{T}] + \mathcal{F}(\Phi).$$ 

If $\phi \in \Homeo_0(\Sigma, \omega)$ has rational asymptotic cycle, this computation shows that $\phi$ has rational rotation direction. The property $\cC(\phi) \in H_1(M_\phi; \mathbb{Q})$ also behaves well under iteration; Lemma~\ref{lem:asymptotic_cycle_iteration} below shows that if $\cC(\phi)$ is rational, then so is $\cC(\phi^k)$ for each $k > 1$. Putting these facts together implies our claim above, that if $\phi$ has periodic Nielsen--Thurston class and $\cC(\phi) \in H_1(M_\phi; \mathbb{Q})$, then it has an iterate $\phi^q \in \Homeo_0(\Sigma, \omega)$ with rational rotation direction. Combining this with Theorem~\ref{thm:main} proves the following theorem.

\begin{thm}\label{thm:non_identity}
Fix a closed surface $\Sigma$ of genus $\geq 2$ and a smooth area form $\omega$. Let $\phi \in \Homeo(\Sigma, \omega)$ be any area-preserving homeomorphism such that $\cC(\phi) \in H_1(M_\phi; \mathbb{Q})$. Then $\phi$ is either periodic or has periodic points of unbounded minimal period. 
\end{thm}

\begin{rem}
Assuming $\cC(\phi) \in H_1(M_\phi; \mathbb{Q})$, for $\phi$ with periodic Nielsen--Thurston class, is sufficient but not necessary for $\phi$ to have an iterate with rational rotation direction. It can be weakened, but we were unable to find a sufficiently elegant condition to write down. It is still true that the set of maps with $\cC(\phi) \in H_1(M_\phi; \mathbb{Q})$ is dense. 
\end{rem} 

\begin{rem}
Theorem~\ref{thm:non_identity} and a similar Baire category argument extend Remark~\ref{rem:infinitely_many} to all area-preserving diffeomorphisms. Fix a closed surface of genus $\geq 2$ and a smooth area form $\omega$. Then a $C^\infty$-generic $\phi \in \Diff(\Sigma, \omega)$ is either periodic or has periodic points of unbounded minimal period. We stress that this statement is not new. It follows from prior work \cite{closing, closing_eh}, which establishes the much stronger statement that a $C^\infty$-generic $\phi \in \Diff(\Sigma, \omega)$ has a dense set of periodic points. 
\end{rem}

There are also asymptotic cycles $\cC(\phi, \mu) \in H_1(M_\phi; \bR)$ for each $\phi$-invariant Borel probability measure $\mu$. The same proof as Theorem~\ref{thm:non_identity}, replacing Theorem~\ref{thm:main} with Theorem~\ref{thm:extension}, implies the following result. 

\begin{thm}\label{thm:non_identity_extension}
Fix a closed surface $\Sigma$ of genus $\geq 2$ and a smooth area form $\omega$. Let $\phi \in \Homeo(\Sigma, \mu)$ be an area-preserving homeomorphism preserving a Borel probability measure $\mu$ of full support. Assume that $\cC(\phi, \mu) \in H_1(M_\phi; \mathbb{Q})$. Then $\phi$ has infinitely many periodic points. Moreover, if $\mu$ has no atoms, then $\phi$ is either periodic or has periodic points of unbounded minimal period. 
\end{thm}

\subsection{Lower genus surfaces}\label{subsec:lower_genus}

The analogues of the above theorems when $\Sigma$ has genus $0$ are already known. Collapsing the boundary components and appealing to \cite{franks_handel, lecalvez_ham} gives a sharp characterization of existence and multiplicity of periodic points. A genus $1$ version of Theorem~\ref{thm:main} follows from work of Le Calvez.

\begin{prop}[\cite{lecalvez_ham}] \label{prop:torus}
Fix a compact surface $\Sigma$ of genus $1$ and a smooth area form $\omega$. Let $\phi \in \Homeo_0(\Sigma, \omega)$ be an area-preserving homeomorphism, and assume that $\mathcal{F}(\phi) \in H_1(\Sigma; \mathbb{Q})\,/\,\Gamma_\omega$. Then $\phi$ is either periodic or has periodic points of unbounded minimal period. 
\end{prop}

The assumptions are stronger than in Theorem~\ref{thm:main}, but the result is sharp. Any translation $(x, y) \mapsto (x + a, y + b)$ on $\mathbb{T}^2 := \bR^2/\mathbb{Z}^2$, where at least one of $a$ and $b$ is not rational, has irrational rotation vector and also has no periodic points. The proof of Proposition~\ref{prop:torus} is straightforward. Any map satisfying the conditions of Proposition~\ref{prop:torus} has a Hamiltonian iterate, and the main result of \cite{lecalvez_ham} shows Hamiltonian torus homeomorphisms are either the identity or have periodic points of unbounded minimal period. A sharp version of Theorem~\ref{thm:non_contractible} for smooth torus maps\footnote{An explanation for why we need smoothness is provided below the statement of Proposition~\ref{prop:non_contractible_torus_technical}.} is also true. 

\begin{thm} \label{thm:non_contractible_torus}
Fix $\phi \in \Diff_0(\mathbb{T}^2, dx \wedge dy)$ and an identity isotopy $\Phi$ such that $\mathcal{F}(\Phi)$ is a real multiple of a rational class. Then $\phi$ is either Hamiltonian, has no periodic points, or has a periodic point which is not $\Phi$-contractible.
\end{thm}

The proof is given in Section~\ref{sec:pfh}. We also remark that if $\mathcal{F}(\phi) \in H_1(\mathbb{T}^2; \mathbb{Q})$, then $\phi$ has a Hamiltonian iterate, so it has a periodic point by Conley--Zehnder's fixed point theorem \cite{conley_zehnder}. There is also a version of Theorem~\ref{thm:non_identity} for the torus.

\begin{prop} \label{prop:torus_non_identity}
Fix a compact surface $\Sigma$ of genus $1$ and a smooth area form $\omega$. Let $\phi \in \Homeo(\Sigma, \omega)$ be an area-preserving homeomorphism, and assume that $\cC(\phi) \in H_1(M_\phi; \mathbb{Q})$. Then $\phi$ is either periodic or has periodic points of unbounded minimal period. 
\end{prop}

Addas Zanata--Tal \cite{azt} proved that an area-preserving torus homeomorphism $\phi$ either has periodic points of unbounded minimal period, is isotopic to a Dehn twist with no periodic points and vertical rotation set reduced to an irrational number, or has an iterate isotopic to the identity. The assumption $\cC(\phi) \in H_1(M_\phi; \mathbb{Q})$ rules out the second case, so we can assume $\phi$ has an iterate $\phi^q$ isotopic to the identity. The assumption $\cC(\phi) \in H_1(M_\phi; \mathbb{Q})$ implies $\phi^q$ has rational rotation direction, and then by Proposition~\ref{prop:torus} $\phi^q$ is either the identity or has periodic points of unbounded minimal period.  

\subsection{Outline} Section \ref{sec:prelim} reviews some important preliminaries. Section \ref{sec:asymptotic_cycles} does some computations of asymptotic cycles which seem standard but which we could not find elsewhere. Section \ref{sec:pfh} presents a brief overview of PFH and the non-vanishing theorem from \cite{closing}, and then proves Theorems~\ref{thm:main}--\ref{thm:non_contractible} and \ref{thm:non_contractible_torus}. Theorems~\ref{thm:non_identity} and \ref{thm:non_identity_extension} were proved above using computations from Section~\ref{sec:asymptotic_cycles} and Theorems~\ref{thm:main} and \ref{thm:extension}. 

\subsection{Acknowledgements} Thanks to Dan Cristofaro-Gardiner for discussions regarding this project. Similar results were proved independently and simultaneously by Guih\'eneuf--Le Calvez--Passeggi \cite{glcp23}, using topological methods (homotopically transverse foliations, forcing theory). I would like to thank them for discussions regarding their work. This work was supported by the National Science Foundation under Award No. DGE-1656466 and the Miller Institute at the University of California Berkeley.

\section{Preliminaries} \label{sec:prelim}

\subsection{Area-preserving maps} 

\subsubsection{Diffeomorphisms}
Write $\Diff(\Sigma)$ for the space of diffeomorphisms $\phi: \Sigma \to \Sigma$, equipped with the topology of $C^\infty$-convergence of maps and their inverses, and let $\Diff(\Sigma, \omega)$ denote the space of diffeomorphisms such that $\phi^*\omega = \omega$. Let $\Diff_0(\Sigma)$ and $\Diff_0(\Sigma, \omega)$ denote the respective connected components of the identity. The group $\Diff_0(\Sigma, \omega)$ contains a large subgroup $\Ham(\Sigma, \omega)$ of \emph{Hamiltonian diffeomorphisms}, the maps with rotation vector $0$. An \emph{isotopy} is a continuous path $\Phi: [0,1] \to \Diff(\Sigma)$. Sometimes we will write $\Phi = \{\phi_t\}_{t \in [0,1]}$ to emphasize our interpretation of $\Phi$ as a one-parameter family of diffeomorphisms. An \emph{identity isotopy} of $\phi \in \Diff(\Sigma)$ is an isotopy $\Phi$ with $\Phi(0) = \text{Id}$ and $\Phi(1) = \phi$. 

\subsubsection{Homeomorphisms}

Write $\Homeo(\Sigma)$ for the space of homeomorphisms $\phi: \Sigma \to \Sigma$, equipped with the topology of $C^0$-convergence of maps and their inverses, and let $\Homeo(\Sigma, \mu)$ denote the space of homeomorphisms preserving a Borel measure $\mu$.  Let $\Homeo_0(\Sigma)$ and $\Homeo_0(\Sigma, \mu)$ denote the respective connected components of the identity. Isotopies of homeomorphisms are defined as above. It is well-known that $\Diff(\Sigma, \omega)$ is $C^0$-dense in $\Homeo(\Sigma, \omega)$, the space of area-preserving homeomorphisms. Write $\overline{\Ham}(\Sigma, \omega) \subset \Homeo_0(\Sigma, \omega)$ for the $C^0$-closure of $\Ham(\Sigma, \omega)$, the group of \emph{Hamiltonian homeomorphisms}. Fathi \cite[Section $6$]{fathi} showed that these are exactly the area-preserving homeomorphisms with rotation vector $0$. 

\subsubsection{Periodic points and orbits} Fix any $\phi \in \Homeo(\Sigma, \omega)$. A \emph{periodic point} of $\phi \in \Homeo(\Sigma, \omega)$ is a point $p \in \Sigma$ such that $\phi^k(p) = p$ for some finite $k \geq 1$, and the \emph{period} of $p$ is the minimal $k$ such that this holds. A \emph{periodic orbit} is a finite set $S = \{x_1, \ldots, x_k\}$ of not necessarily distinct points in $\Sigma$ which are cyclically permuted by $\phi$. A periodic orbit is \emph{simple} if all of the points are distinct. 

Fix $\phi \in \Homeo_0(\Sigma, \omega)$ and an identity isotopy $\Phi$. Fix any periodic point $p$ of period $k \geq 1$. The union of arcs 
$$\gamma_p := \bigcup_{j=0}^{k-1} \{\phi_t(\phi^j(p))\}_{t \in [0,1]}$$ 
is a closed loop in $\Sigma$. The point $p$ is \emph{$\Phi$-contractible} if $\gamma_p$ is contractible. Note that if $\Sigma$ has genus $\geq 2$, $\Homeo_0(\Sigma, \omega)$ is simply connected, so $\Phi$-contractibility is independent of the choice of $\Phi$. We will not specify $\Phi$ in this case. 

\subsection{Rotation vectors}

\subsubsection{Definition}

Fix $\phi \in \Homeo_0(\Sigma)$ and an identity isotopy $\Phi$. To any $\phi$-invariant Borel probability measure $\mu$ we associate a class $\mathcal{F}(\Phi, \mu) \in H_1(\Sigma; \bR)$ called its \emph{rotation vector}. The rotation vector depends only on the homotopy class of $\Phi$ relative to its endpoints. When $\Sigma$ has genus $\geq 2$, the space $\Homeo_0(\Sigma)$ is simply connected, so $\mathcal{F}(\Phi, \mu)$ is independent of the choice of $\Phi$, and we sometimes write it as $\mathcal{F}(\phi, \mu)$ instead. 

Let $[\Sigma, \mathbb{T}]$ denote the group of continuous maps from $\Sigma$ to the circle $\mathbb{T} := \mathbb{R}/\mathbb{Z}$. This is isomorphic to $H^1(\Sigma; \mathbb{Z})$. The isomorphism sends a class $[f] \in [\Sigma, \mathbb{T}]$ to the pullback $f^*d\theta$ of the oriented generator $d\theta \in H^1(\mathbb{T}; \mathbb{Z})$. The universal coefficient theorem implies $H_1(\Sigma; \mathbb{R})$ is isomorphic to $\Hom([\Sigma, \mathbb{T}], \bR)$. 

For any continuous circle-valued function $f: \Sigma \to \mathbb{T}$, the isotopy $\Phi$ defines a real-valued lift $g: \Sigma \to \mathbb{R}$ of the null-homotopic $\mathbb{T}$-valued function $f \circ \phi - f$. The map 
$$f \mapsto \int_\Sigma g\,d\mu$$
defines a real-valued linear functional on $[\Sigma, \mathbb{T}]$, and the rotation vector $\mathcal{F}(\Phi, \mu) \in H_1(\Sigma; \bR)$ is the associated homology class. 

\subsubsection{Rotation vectors of periodic orbits} 

Any periodic orbit $S = \{x_1, \ldots, x_k\}$ determines an invariant Borel probability measure; the average of the $\delta$-measures at its points. The rotation vector $\mathcal{F}(\Phi, S)$ is the rotation vector of this measure. This has a nice geometric interpretation. The union of arcs 
$$\gamma_S := \bigcup_{j=1}^{k} \{\phi_t(x_j)\}_{t \in [0,1]}$$ 
is a closed, oriented loop in $\Sigma$. It is easy to show
\begin{equation} \label{eq:rotation_vector_orbit} k^{-1} \cdot \mathcal{F}(\Phi, S) = [\gamma_S]. \end{equation}

If $p \in \Sigma$ is a periodic point, its rotation vector $\mathcal{F}(\Phi, p)$ is defined to be the rotation vector of any simple periodic orbit containing $p$. The identity \eqref{eq:rotation_vector_orbit} shows that if $\mathcal{F}(\Phi, p) \neq 0$ then $p$ is not $\Phi$-contractible. 

\subsubsection{Rotation vectors of area-preserving homeomorphisms} 

Fix $\phi \in \Homeo_0(\Sigma, \omega)$ and any identity isotopy $\Phi$. We use the following notation for the rotation vector of the normalized area measure:
$$\mathcal{F}(\Phi) := \mathcal{F}(\Phi, \Big(\int_\Sigma \omega\Big)^{-1}\cdot \omega) \in H_1(\Sigma; \bZ).$$

This invariant was introduced by Fathi \cite{fathi} as the \emph{mass flow}. The function $\mathcal{F}$ is $C^0$-continuous in $\Phi$ and additive with respect to pointwise composition of isotopies. Moreover, the image $\Gamma_\omega := \mathcal{F}(\pi_1(\Homeo(\Sigma, \omega))$ of the subgroup of loops based at the identity is a lattice in $H_1(\Sigma; \mathbb{Z})$. If $\Sigma = \mathbb{T}^2$ then $\Gamma_\omega = H_1(\Sigma; \mathbb{Z})$. If $\Sigma$ is closed and has genus $\geq 2$ then $\Gamma_\omega = \{0\}$. We end up with a homomorphism
 $$\mathcal{F}: \Homeo_0(\Sigma, \omega) \to H_1(\Sigma; \mathbb{R})\,/\,\Gamma_\omega.$$

\begin{rem}
For smooth area-preserving maps, the rotation vector is Poincar\'e dual to the \emph{flux homomorphism}, an invariant which might be more familiar to symplectic geometers. 
\end{rem}

\subsection{Mapping torii} Fix any $\phi \in \Homeo(\Sigma)$. The \emph{mapping torus} of $\phi$ is the compact $3$-manifold $M_\phi$ defined by quotienting $\mathbb{R}_t \times \Sigma$ by the relation $(1, p) \sim (0, \phi(p))$. Translation in the $t$-direction with speed $1$ yields a continuous flow $\{\psi^t_R\}_{t \in \bR}$ on $M_\phi$ called the \emph{suspension flow}. Its closed integral curves are in one-to-one correspondence with simple periodic orbits of $\phi$. If $\phi$ preserves a Borel measure $\mu$, the suspension flow preserves the measure $dt \otimes \mu$ on $M_\phi$. 

Suppose $\phi \in \Homeo_0(\Sigma)$. Choose an identity isotopy $\Phi = \{\phi_t\}_{t \in [0,1]}$. This choice defines a homeomorphism
\begin{equation} \label{eq:mt_iso} \eta: \mathbb{T} \times \Sigma \to M_\phi \end{equation}
by the map $[(t, p)] \mapsto [(t, \phi_t^{-1}(p))]$. This homeomorphism provides a useful method for recovering the rotation vector of a periodic orbit. Let $S = \{x_1, \ldots, x_k\} \subset \Sigma$ be a simple periodic orbit, and let $\gamma \subset M_\phi$ be the associated closed integral curve of the suspension flow. Using $\eta$ to realize $\gamma$ as a loop in $\mathbb{T} \times \Sigma$, its homology class is easily computed:
\begin{equation} \label{eq:mt_rotation_vectors} k^{-1} \cdot [\gamma] = [\mathbb{T}] + \mathcal{F}(\Phi, S) \in H_1(\mathbb{T} \times \Sigma; \bZ).\end{equation}

\section{Asymptotic cycles} \label{sec:asymptotic_cycles} This section discusses the asymptotic cycle construction, and does several useful computations. Fix a compact surface $\Sigma$, a map $\phi \in \Homeo(\Sigma)$, and a $\phi$-invariant Borel probability measure $\mu$. The \emph{asymptotic cycle}, which was introduced by Schwartzman \cite{schwartzman}, is a homology class $\cC(\phi, \mu) \in H_1(M_\phi; \bR)$. If $\phi$ is area-preserving, $\cC(\phi)$ denotes the asymptotic cycle of the normalized area measure. 

\subsection{Definition} We define $\cC(\phi, \mu)$ as a real-valued linear functional on $[M_\phi, \mathbb{T}]$. Fix any continuous $f: M_\phi \to \mathbb{T}$. For each $s \in \bR$, write $f_s := f \circ \psi^s_R$. The functions $f_s - f$ are a continuous family of null-homotopic circle-valued functions, so they lift to a unique continuous family $\{g_s\}_{s \in \bR}$ of functions $M_\phi \to \bR$ with $g_0 \equiv 0$. Kingman's subadditive ergodic theorem implies that $G := \lim_{s \to \infty} g_s/s$ is a well-defined $(dt \otimes \mu)$-integrable function. We set $\langle \cC(\phi, \mu), f \rangle$ to be the integral of $G$. This is linear and homotopy-invariant in $f$ (see \cite{schwartzman}), so it defines a real-valued linear functional on $H^1(M_\phi; \bZ)$, and therefore defines a class in $H_1(M_\phi; \bR)$. 

\begin{rem}
Fix any compact manifold $M$. An asymptotic cycle, taking values in $H_1(M; \mathbb{R})$, is defined as above for any choice of a continuous flow $\{\psi^t\}_{t \in \bR}$ and a $\psi$-invariant Borel probability measure. 
\end{rem}

\subsection{Maps isotopic to the identity} When $\phi \in \Homeo_0(\Sigma)$, the rotation vector of $\mu$ can be recovered from $\cC(\phi, \mu)$. 

\begin{lem} \label{lem:asymptotic_cycle_computation}
    Fix a compact surface $\Sigma$, a Borel probability measure $\mu$, and $\phi \in \Homeo_0(\Sigma, \mu)$. For any identity isotopy $\Phi$, the pullback of $\cC(\phi, \mu)$ by \eqref{eq:mt_iso} satisfies the identity 
\begin{equation} \label{eq:asymptotic_cycle_computation} \eta^* \cdot \cC(\phi, \mu) = [\mathbb{T}] + \mathcal{F}(\Phi, \mu) \in H_1(\mathbb{T} \times \Sigma; \mathbb{R}).\end{equation} 
\end{lem}

\begin{proof}
	Write $\Phi = \{\phi_t\}_{t \in [0,1]}$ and write $q_t := \eta^{-1} \circ \psi^t_R \circ \eta$, where we recall $\{\psi^t_R\}_{t \in \bR}$ is the suspension flow. Write $\mu_t := (\phi_t)_*(\mu)$ for every $t$. The pullback $\eta^* \cdot \cC(\phi, \mu)$ is the asymptotic cycle of the flow $\{q_t\}_{t \in \bR}$ with respect to $dt \otimes \mu_t = \eta^*(dt \otimes \mu)$. We extend the isotopy to a map $\Phi: \mathbb{R} \to \Homeo_0(\Sigma)$ by setting $\phi_t := \phi_{t - \lfloor t \rfloor}\phi^{\lfloor t \rfloor}$.For any $s \in \mathbb{R}$ and $t \in [0,1)$, we compute $q_s(t, p) = (s + t, \phi_{s + t}\phi_{t}^{-1}(p)) \in \mathbb{T} \times \Sigma$. 

The $\mathbb{T}$-invariant functions and the projection $\pi: \mathbb{T} \times \Sigma \to \mathbb{T}$ define a basis of $H^1(\mathbb{T} \times \Sigma; \bZ)$. The lemma is proved by showing
\begin{equation} \label{eq:asymptotic_cycle_computation_2} \langle \eta^* \cdot \cC(\phi, \mu), \pi \rangle = 1,\qquad\qquad \langle \eta^* \cdot \cC(\phi, \mu), f \rangle = \langle \mathcal{F}(\Phi, \mu), f \rangle\end{equation} 
for any $\mathbb{T}$-invariant $f: \mathbb{T} \times \Sigma \to \mathbb{T}$. The real-valued lift $g_s$ of $\pi \circ q_s - \pi$ is $g_s(t, p) = s$, so the integral of $g_s/s$ is always $1$. This proves the first identity in \eqref{eq:asymptotic_cycle_computation_2}. 

Fix any $\mathbb{T}$-invariant $f: \mathbb{T} \times \Sigma \to \mathbb{T}$.  Write $f_s := f \circ q_s$ and let $\{g_s\}_{s \in \bR}$ be the real-valued lift of the family $\{f_s - f\}_{s \in \bR}$. Fix any $t \in [0,1)$ and set $\phi^{(t)} \in \Homeo_0(\Sigma, \mu_t)$ to be the conjugate of $\phi$ by $\phi_t$. The function $(f_\tau - f)(t, -)$ is the displacement of $f(t, -)$ under the identity isotopy $\Phi^{t} := \{\phi_{\tau + t}\phi_t^{-1}\}_{\tau \in [0,1]}$ ending at $\phi^{(t)}$. It follows that the integral of $g_1(t, -)$ with respect to $\mu_t$ is $\mathcal{F}(\Phi^t, \mu_t)$.  We note $g_s(t, -) = g_{s - 1}(t, -) + g_1(t, -) \circ \phi^{(t)}$, so the $\mu_t$-integral of $g_s/s$ for $s \in \mathbb{N}$ is the same as the integral of $g_1$. The isotopy $\Phi^t$ is homotopic relative to its endpoints to the conjugated isotopy $\phi_t \circ \Phi \circ \phi_t^{-1}$, so $\mathcal{F}(\Phi^t, \mu_t) = \mathcal{F}(\Phi, \mu)$. Integrate over $t \in [0,1)$ and use the fact that $f$ is $\mathbb{T}$-invariant to show
$$\frac{1}{s}\int_0^1\big(\int_\Sigma g_s(t, -)\,d\mu_t\big)\,dt = \int_0^1 \langle \mathcal{F}(\Phi, \mu), f(t, -) \rangle\,dt = \langle \mathcal{F}(\Phi, \mu), f \rangle$$
for any $s \in \mathbb{N}$. This proves the second identity in \eqref{eq:asymptotic_cycle_computation_2}. 
\end{proof}

\subsection{Smooth maps} When $\phi \in \Diff(\Sigma, \omega)$, the area form $\omega$ defines a closed $2$-form $\omega_\phi$ on the mapping torus $M_\phi$, which restricts to $0$ on the boundary. Its cohomology class is Poincar\'e dual to the asymptotic cycle. 

\begin{lem} \label{lem:rational_pd}
    Assume $\phi \in \Diff(\Sigma, \omega)$. Then $[\omega_\phi] \in H^2(M_\phi, \partial M_\phi; \mathbb{R})$ is Poincar\'e dual to $\Big(\int_\Sigma \omega\Big) \cdot \cC(\phi)$. 
\end{lem}

\begin{proof}
	Let $d\theta$ denote the closed one-form on $\mathbb{T}$ with integral $1$. If $\phi$ is smooth, then a circle-valued function $f: M_\phi \to \mathbb{T}$ corresponds to $[f^*d\theta] \in H^1(M_\phi; \bR)$. The lemma follows from showing
    \begin{equation} \label{eq:rational_pd_1} \int_{M_\phi} f^*d\theta \wedge \omega_\phi = \Big(\int_\Sigma \omega\Big) \cdot \langle \cC(\phi), f \rangle \end{equation}
    for any $f: M_\phi \to \mathbb{T}$. Write $f_s = f \circ \psi^s_R$ for each $s \in \bR$, and let $\dot f_s = (f_s^*d\theta)(R): M_\phi \to \bR$ denote the time derivative. The associated real-valued lifts are $g_s := \int_0^s \dot f_\tau d\tau$ for $s > 0$. For any $s > 0$, $g_{2s} = g_{s} + g_{s} \circ \psi^{s}_R$. Since $dt \wedge \omega_\phi$ is $R$-invariant, this implies that the integral of $g_{2s}/2s$ over $M_\phi$ is equal to the integral of $g_{s}/s$. Repeated division by $2$ shows
\begin{align*} \int_{M_\phi} \frac{1}{s}g_s\,dt\wedge \omega_\phi &= \lim_{\tau \to 0} \int_{M_\phi} \frac{1}{\tau}g_\tau\,dt\wedge\omega_\phi = \int_{M_\phi} \dot f_0\,dt \wedge \omega_\phi \\
&= \int_{M_\phi} (f^*d\theta)(R)\,dt\wedge\omega_\phi = \int_{M_\phi} f^*d\theta \wedge \omega_\phi
\end{align*}
for any $s > 0$. This proves \eqref{eq:rational_pd_1}.  
\end{proof}

\subsection{Behavior under iteration} We show that rationality of $\cC(\phi, \mu)$ is preserved under iteration. 

\begin{lem}\label{lem:asymptotic_cycle_iteration}
Fix any $k \in \mathbb{N}$. If $\cC(\phi, \mu) \in H_1(M_\phi; \mathbb{Q})$, then $\cC(\phi^k, \mu) \in H_1(M_{\phi^k}; \mathbb{Q})$. 
\end{lem}

\begin{proof}
There is a covering map $\pi_k: M_{\phi^k} \to M_\phi$ with deck group $\mathbb{Z}/k\mathbb{Z}$, given by the map $[(t,p)] \mapsto [(kt - \lfloor kt \rfloor, \phi^{\lfloor kt \rfloor}(p))]$. The deck group is generated by the map $T: [(t, p)] \mapsto [(t - 1/k, \phi(p))]$. Denote the suspension flows of $\phi$ and $\phi^k$ by $\{\psi^t\}$, $\{\psi^t_k\}$ respectively. The group $[M_{\phi^k}, \mathbb{T}]$ is spanned by functions of the form $f \circ T - f$ and those which are pulled back by $\pi_k$ from $M_\phi$. The suspension flow of $\phi^k$ commutes with the covering translations, so if $\{g_s\}_{s \in \bR}$ denotes the real-valued lifts of the family $\{f \circ \psi^s_k - f\}_{s \in \bR}$, then $\{g_s \circ T\}_{s \in \bR}$ are the real-valued lifts of $\{f \circ T \circ \psi^s_k - f \circ T\}_{s \in \bR}$. The map $T$ preserves $dt \otimes \mu$, so $g_s \circ T$ and $g_s$ have the same integral. We conclude
$$\langle \cC(\phi^k), f \circ T - f \rangle = 0.$$

It remains to consider functions pulled back from $M_\phi$. Note that for any $f: M_\phi \to \mathbb{T}$, the pairing $\langle \cC(\phi^k, \mu), f \circ \pi_k \rangle$ is equal to $\langle (\pi_k)_* \cdot \cC(\phi^k, \mu), f \rangle$. Since we are assuming $\cC(\phi)$ is rational, the rationality of $\cC(\phi^k)$ therefore follows from the identity
\begin{equation} \label{eq:asymptotic_cycle_iteration} (\pi_k)_* \cdot  \cC(\phi^k, \mu) = k \cdot \cC(\phi, \mu). \end{equation}

The key observation here is the commutation relation $\pi_k \circ \psi_k^{t/k} = \psi^t \circ \pi_k$. This shows that the asymptotic cycle of the flow $\{\psi_k^{t/k}\}_{t \in \bR}$ pushes forward to $\cC(\phi)$. Rescaling a flow in time by a factor of $\lambda$ also multiplies the asymptotic cycle by $\lambda$. The asymptotic cycle of $\{\psi_k^{t/k}\}_{t \in \bR}$ is $k^{-1} \cdot \cC(\phi^k)$, so this proves \eqref{eq:asymptotic_cycle_iteration}. 
\end{proof}

\section{PFH and proofs of main theorems} \label{sec:pfh}

\subsection{Overview of PFH and non-vanishing}
Fix a closed surface $\Sigma$ and a smooth area form $\omega$. Fix an area-preserving diffeomorphism $\phi \in \Diff(\Sigma, \omega)$. The area form $\omega$ defines a closed 2-form on the mapping torus $M_\phi$, denoted by $\omega_\phi$. Let $t$ be the coordinate for the interval component of $[0,1] \times \Sigma$. Then $dt$ pushes forward to a smooth 1-form on $M_\phi$. The pair $(dt,\omega_\phi)$ forms a stable Hamiltonian structure on $M_\phi$, and the Reeb vector field is a smooth vector field $R$ generating the suspension flow $\{\psi^t_R\}_{t \in \bR}$. The associated two-plane bundle $\ker(dt)$ is equal to the vertical tangent bundle of the fibration $M_\phi \to \mathbb{T}$, which we denote by $V$. 

Fix some nonzero homology class $\Gamma \in H_1(M_\phi; \bZ)$. The \emph{PFH generators} are
 finite sets
$\Theta = \{(\gamma_i, m_i)\}$
of pairs of embedded Reeb orbits $\gamma_i$ and multiplicities $m_i \in \mathbb{N}$ which satisfy the following three conditions: (1) the orbits $\gamma_i$ are distinct, (2) the multiplicity $m_i$ is $1$ whenever $\gamma_i$ is a hyperbolic orbit, (3) $\sum_i m_i[\gamma_i] = \Gamma$. 
The chain complex $\PFC_*(\phi, \Gamma)$ is the free module over a commutative coefficient ring\footnote{This can be anything when $\Gamma$ solves \eqref{eq:gammad} for some $d$, but in general $\Lambda$ must be a Novikov ring.} $\Lambda$ generated by the set of all PFH generators. 

The differential on $\PFC_*(\phi, \Gamma)$ counts ``ECH index 1" $J$-holomorphic currents between orbit sets. The homology of this chain complex is the \emph{periodic Floer homology} $\PFH_*(\phi, \Gamma)$. This homology theory was constructed by Hutchings \cite{Hutchings02}; see \cite{HutchingsNotes} for a detailed exposition of the closely related theory of \emph{embedded contact homology}.  The PFH group depends only on the Hamiltonian isotopy class of $\phi$; this allows us to define PFH for a degenerate map as the PFH of any sufficiently close nondegenerate Hamiltonian perturbation. We now precisely state the non-vanishing theorem for PFH. 
 
\begin{prop}[{\cite[Theorem $1.4$]{closing}}] \label{prop:pfh_nonvanishing} Fix a closed surface $\Sigma$ of any genus $g$ and a smooth area form $\omega$. Fix any area-preserving diffeomorphism $\phi \in \Diff(\Sigma, \omega)$. Then for any $d > \max(2g - 2, 0)$ and any class $\Gamma \in H_1(M_\phi; \mathbb{Z})$ satisfying 
\begin{equation} \label{eq:gammad} \text{PD}(\Gamma) = \Big(\int_\Sigma \omega\Big)^{-1}(d + 1 - g)[\omega_\phi] - \frac{1}{2}c_1(V)\end{equation}
the group $PFH(\phi, \Gamma)$ with $\mathbb{Z}/2$-coefficients is nonzero. 
\end{prop}

The result as stated in \cite{closing} only asserts non-vanishing for $d$ sufficiently large, but the explicit lower bound is not difficult to extract once the details are understood. We only need $d$ large enough to ensure that $\PFH$ is isomorphic to the ``bar'' version $\overline{\text{HM}}$ of monopole Floer homology. Lee--Taubes \cite[Theorem $1.2$, Corollary $1.5$]{lee_taubes} prove this isomorphism assuming $d > \max(2g - 2, 0)$.  The non-vanishing theorem is a key ingredient in the following technical result.

\begin{prop}\label{prop:non_contractible_technical}
Fix a closed surface $\Sigma$ of genus $g \geq 2$ and a smooth area form $\omega$. Let $\phi \in \Homeo_0(\Sigma, \omega)$ be such that there exists nonzero $h \in H_1(\Sigma; \bZ)$ and a positive real number $c > 0$ satisfying $\mathcal{F}(\phi) = c \cdot h$. Then for any rational number $p/q \in (0, c]$ with $q > g - 1$, $\phi$ has a non-contractible periodic point with minimal period $\leq q + g - 1$. 
\end{prop}

The same argument proves an analogue for smooth torus maps, but we need to rule out maps without fixed points in the statement. 

\begin{prop}\label{prop:non_contractible_torus_technical}
Fix $\phi \in \Diff_0(\mathbb{T}^2, dx \wedge dy)$. Assume that $\phi$ has at least one fixed point. Assume further that there exists an identity isotopy $\Phi$, nonzero $h \in H_1(\Sigma; \bZ)$ and a positive real number $c > 0$ satisfying $\mathcal{F}(\Phi) = c \cdot h$. Then for any rational number $p/q \in (0, c]$, $\phi$ has a $\Phi$-non-contractible periodic point with minimal period $\leq q$. 
\end{prop}

\begin{rem}
We cannot extend Proposition~\ref{prop:non_contractible_torus_technical} to homeomorphisms since the blow-up argument requires the map to be differentiable, and it is not clear that a torus homeomorphism with a fixed point can be approximated by diffeomorphisms with fixed points and the same rotation vector. 
\end{rem}

Theorem~\ref{thm:non_contractible} and Theorem~\ref{thm:non_contractible_torus} from the introduction respectively follow from Proposition~\ref{prop:non_contractible_technical} and Proposition~\ref{prop:non_contractible_torus_technical}. We now outline the plan for the rest of the section. Section~\ref{subsec:main_proof} proves Theorems~\ref{thm:main} and \ref{thm:extension} assuming Theorem~\ref{thm:non_contractible}. Section~\ref{subsec:rotation_vectors_proof} proves Proposition~\ref{prop:non_contractible_technical}. Section \ref{subsec:torus_proof} proves Proposition~\ref{prop:non_contractible_torus_technical}. 

\subsection{Existence of infinitely many periodic points} \label{subsec:main_proof}

We prove Theorems~\ref{thm:main} and \ref{thm:extension} using Theorem~\ref{thm:non_contractible}. We assume that $\Sigma$ is a closed surface of genus $\geq 2$, since we can reduce to this case by collapsing the boundary components. 

\subsubsection{Proof of Theorem~\ref{thm:main}} Fix any $\phi \in \Homeo_0(\Sigma, \omega)$ with rational rotation direction. Le Calvez \cite{lecalvez_ham} showed any Hamiltonian homeomorphism on a surface of genus $\geq 1$ is either the identity or has periodic points of unbounded minimal period. Therefore we only consider the case where $\phi$ is not Hamiltonian, in which case it has a non-contractible periodic point by Theorem~\ref{thm:non_contractible}. The arguments from \cite[Section $4$]{lecalvez} then show that it has periodic points of unbounded minimal period. 

Here is a high-level outline of \cite[Section $4$]{lecalvez}. Write $\widetilde{\phi}$ for the lift of $\phi$ to the universal cover $\widetilde{\Sigma}$ commuting with the covering translations. The non-contractible periodic point $p$, which we assume to have minimal period $k$, lifts to a point $\widetilde{p}$ such that $\widetilde{\phi}^k(\widetilde{p}) = T \cdot \widetilde{p}$ for some $T \in \pi_1(\Sigma)$. Pass to the annular cover $\widetilde{\Sigma} / T$ and compactify to produce a homeomorphism $\hat{\phi}$ of the closed strip $[0,1] \times \mathbb{R}$ with rotation interval containing $[0, 1/k]$. Le Calvez's refinement of the Poincar\'e--Birkhoff--Franks theorem \cite[Theorem $2.4$]{lecalvez} then shows that either $\phi$ has periodic points of unbounded minimal period or $\hat{\phi}$ does not satisfy the intersection property. In this latter case, a forcing argument is used to produce periodic points of unbounded minimal period regardless. 

\subsubsection{Proof of Theorem~\ref{thm:extension}} Fix a map $\phi \in \Homeo_0(\Sigma, \mu)$, where $\mu$ is a Borel probability measure of full support with $\mu(\partial\Sigma) = 0$ and $\mathcal{F}(\phi, \mu)$ is a real multiple of a rational class. We may assume without loss of generality that $\Sigma$ is closed by collapsing the boundary.

There exists $t \in [0,1]$ and a unique decomposition (see \cite{johnson_atomic}) 
$$\mu = t\mu_0 + (1 - t)\mu_1$$ with the following properties. The measures $\mu_0$ and $\mu_1$ are Borel probability measures, $\mu_0$ has no atoms, $\mu_1$ is purely atomic, and they are mutually ``S-singular''. This means that for any Borel set $E \subset \Sigma$, 
$$\mu_0(E) = \sup\{\mu_0(E \cap F)\,|\,\mu_1(F) = 0\},\qquad \mu_1(E) = \sup\{\mu_1(E \cap F)\,|\,\mu_0(F) = 0\}.$$

The decomposition $\mu = t\phi_*\mu_0 + (1 - t)\phi_*\mu_1$ also satisfies these properties, so the uniqueness property implies that both $\mu_0$ and $\mu_1$ are $\phi$-invariant. 

Since $\mu_1$ is $\phi$-invariant, each of its atoms are periodic points. If $\mu_1$ has infinite support, then there are infinitely many periodic points. If $\mu_1$ has nonzero rotation vector, then $\phi$ has a non-contractible periodic point, and therefore has periodic points of unbounded minimal period by the argument of Le Calvez mentioned above. If $\mu_1$ has finite support and zero rotation number, then $\mu_0$ has full support and rotation number proportional to a rational class. Since it is atomless and has full support, \cite[Theorem $2_1$]{oxtoby_ulam} shows it is homeomorphic to a smooth area measure, so $\phi$ is conjugate to an area-preserving homeomorphism with rational rotation direction, and we apply Theorem~\ref{thm:main}. 

\subsection{Proof of Proposition~\ref{prop:non_contractible_technical}} \label{subsec:rotation_vectors_proof}

\subsubsection{Information from PFH} The following lemma records the relevant information needed from Proposition~\ref{prop:pfh_nonvanishing}. We only state and prove it for diffeomorphisms, but remark that it also holds for homeomorphisms by an approximation argument.

\begin{lem} \label{lem:pfh_nonvanishing}
Fix a closed surface $\Sigma$ of genus $g$. Fix $\phi \in \Diff_0(\Sigma, \omega)$ and an identity isotopy $\Phi$, and assume $\mathcal{F}(\Phi) \in H_1(\Sigma; \mathbb{Q})$. Let $d$ be the smallest integer greater than $\max(2g - 2, 0)$ such that
$$(d + 1 - g) \cdot \mathcal{F}(\Phi) \in H_1(\Sigma; \bZ).$$

Then there exists a set of simple periodic orbits $\{S_i\}_{i=1}^N$ of periods $\{k_i\}_{i=1}^N$ and a set of positive integers $\{m_i\}_{i=1}^N$ such that 
\begin{equation} \label{eq:pfh_nonvanishing} \sum_{i=1}^N m_i k_i = d,\qquad \sum_{i=1}^N m_i k_i \mathcal{F}(\Phi, S_i) = (d + 1 - g) \cdot \mathcal{F}(\Phi).\end{equation}
\end{lem}

\begin{proof}
Using Lemmas~\ref{lem:asymptotic_cycle_computation} and \ref{lem:rational_pd}, we compute
$$\eta^*[\omega_\phi] = \Big(\int_\Sigma \omega\Big) \cdot \text{PD}([\mathbb{T}] + \mathcal{F}(\Phi)) \in H^2(\mathbb{T} \times \Sigma; \bR).$$

The class
\begin{equation*} \Gamma = d[\mathbb{T}] + (d + 1 - g)\mathcal{F}(\Phi). \end{equation*}
solves \eqref{eq:gammad}. By Proposition~\ref{prop:pfh_nonvanishing}, there exists an orbit set $\Theta = \{(\gamma_i, m_i)\}$ such that $\sum_i m_i[\gamma_i] = \Gamma$. For each $i$, let $S_i$ be the simple periodic orbit of $\phi$ corresponding to $\gamma_i$. Sum up the homology class computation \eqref{eq:mt_rotation_vectors} over all $i$ to conclude 
$$d[\mathbb{T}] + (d + 1 - g)\mathcal{F}(\Phi) = \sum_i m_ik_i([\mathbb{T}] + \mathcal{F}(\Phi, S_i)).$$
This identity implies \eqref{eq:pfh_nonvanishing}. 
\end{proof}

\subsubsection{Blow-up} Fix a closed surface $\Sigma$ of genus $\geq 2$, a smooth area form $\omega$ of area $A$, and a diffeomorphism $\phi \in \Diff_0(\Sigma, \omega)$. Suppose that $\phi$ has a contractible fixed point $p$. Choose an identity isotopy $\Phi$ which fixes $p$ (one always exists, see \cite[Proposition $9$]{lrhs16}). We give a precise account here of how to blow up the fixed point $p$ and cap it with a disk of any prescribed area. 

Fix polar coordinates $(r, \theta)$ on $\bR^2$, and for any $s > 0$ denote by $D_s := \{0 \leq r < s\}$ and $A_s := D_s \setminus \{0\}$ the open disk and punctured disk, respectively, of radius $s$ centered at the origin. Write $\dot\Sigma := \Sigma\,\setminus\,\{p\}$. For positive $\delta \ll 1$, there is a symplectic embedding $\iota: (A_\delta, r dr \wedge d\theta) \hookrightarrow (\Sigma, \omega)$ which is a diffeomorphism onto a punctured neighborhood of $p$. Next, fix a parameter $B > A$, which will be the area of the capped surface, and fix $s_1 > s_0 > 0$ such that $B - A = \pi s_0^2$ and $s_1^2 - s_0^2 = \delta^2$. Then the map $(r, \theta) \mapsto (\sqrt{s_0^2 + r^2}, \theta)$ is a symplectic embedding $\tau: (A_\delta, r dr \wedge d\theta) \hookrightarrow (D_{s_1}, dx \wedge dy)$ which identifies $A_\delta$ with the annulus $\{s_0 < r < s_1\}$. The surface $\widehat{\Sigma}$ is the surface constructed by gluing $\dot\Sigma$ and $D_{s_1}$ along $A_\delta$, using the symplectic embeddings $\iota$ and $\tau$. The glued surface $\widehat{\Sigma}$ has a symplectic form $\hat{\omega}$ restricting to $\omega$ on $\Sigma$ and $dx \wedge dy$ on $D_{s_1}$. The area of $\widehat{\Sigma}$ is $A + \pi s_0^2 = B$ as desired. 

The isotopy $\Phi = \{\phi^t\}_{t \in [0,1]}$ extends to an identity isotopy $\hat{\Phi}: [0,1] \to \Diff(\widehat{\Sigma}, \hat{\omega})$. Since the isotopy fixes $p$, it coincides with a Hamiltonian isotopy in a neighborhood of $p$; extend the generating Hamiltonian to $\widehat{\Sigma}$ to produce the desired extension. The following lemma computes the rotation vector of the extension. 

\begin{lem} \label{lem:blowup_rotation}
Fix a closed surface $\Sigma$, an area form $\omega$, and a point $p \in \Sigma$. Let $\Phi: [0,1] \to \Diff_0(\Sigma, \omega)$ be an identity isotopy such that $\Phi(t)$ fixes $p \in \Sigma$ for each $t$. Fix any extension $\hat{\Phi}$ of $\Phi$ to the blown-up surface $(\widehat{\Sigma}, \hat{\omega})$, and let $\pi: \hat{\Sigma} \to \Sigma$ be the blow-down map. Then the rotation vector of $\hat{\Phi}$ satisfies the identity
\begin{equation} \label{eq:blowup_rotation}
\Big(\int_{\widehat{\Sigma}}\hat{\omega}\Big) \cdot \pi_*\mathcal{F}(\hat{\Phi}) = \Big(\int_\Sigma \omega\Big) \cdot \mathcal{F}(\Phi).
\end{equation}
\end{lem}

\begin{proof}
Fix any $f: \Sigma \to \mathbb{T}$ and set $\hat{f} := f \circ \pi: \widehat{\Sigma} \to \mathbb{T}$. Write $g: \Sigma \to \bR$ and $\hat{g}: \widehat{\Sigma} \to \bR$ for the lifts of $f \circ \phi - f$ and $\hat{f} \circ \hat{\phi} - \hat{f}$ induced by the isotopies $\Phi$ and $\hat{\Phi}$. The function $\hat{f}$ equals $f$ on $\dot\Sigma \subset \widehat{\Sigma}$ and is constant on its complement. Both sets are $\hat{\Phi}$-invariant, so $\hat{g} = g$ on $\dot\Sigma$ and $\hat{g} = 0$ elsewhere. We conclude
$$\Big(\int_{\widehat{\Sigma}}\hat{\omega}\Big) \cdot \langle \mathcal{F}(\hat{\Phi}), \hat{f} \rangle = \int_{\widehat{\Sigma}} \hat{g}\,\hat{\omega} = \int_{\dot\Sigma} g\,\omega = \Big(\int_\Sigma \omega\Big) \cdot \langle \mathcal{F}(\Phi), f \rangle$$
which implies \eqref{eq:blowup_rotation}. 
\end{proof}

\subsubsection{Proof for diffeomorphisms} Fix a closed surface $\Sigma$ of genus $\geq 2$, a smooth area form $\omega$ of area $A$, and $\phi \in \Diff_0(\Sigma, \omega)$ such that $\mathcal{F}(\phi) = c \cdot h$ where $c > 0$ and $h \in H_1(\Sigma; \mathbb{Z})$. Choose a rational number $p/q \in (0, c]$ with $q > g - 1$. Fix $B > A$ such that $A/B = p/q \cdot c^{-1}$. Blow up the fixed point $p$ to get a surface $\widehat{\Sigma}$ of area $B$, and extend the identity isotopy $\Phi$ to an identity isotopy $\hat{\Phi}$ with endpoint $\hat{\phi} \in \Diff_0(\widehat{\Sigma}, \hat{\omega})$. By \eqref{eq:blowup_rotation}, 
$$\pi_*\mathcal{F}(\hat{\phi}) = A/B \cdot \mathcal{F}(\phi) = p/q \cdot h.$$

The map $\pi_*$ is an isomorphism $H_1(\widehat{\Sigma}; \bZ) \simeq H_1(\Sigma; \bZ)$, so we conclude that $q\mathcal{F}(\hat{\phi}) \in H_1(\widehat{\Sigma}; \bZ)$. By Lemma~\ref{lem:pfh_nonvanishing}, $\hat{\phi}$ has a non-contractible periodic point $z$ of minimal period $\leq q + g - 1$. This periodic point must lie in $\dot\Sigma \subset \widehat{\Sigma}$. This is because the complement $\widehat{\Sigma}\,\setminus\,\dot\Sigma$ is a $\hat{\Phi}$-invariant disk, so any periodic point contained in it is contractible. The point $z$ is therefore a non-contractible periodic point of $\phi$ of minimal period $\leq q + g - 1$. 

\subsubsection{Proof for homeomorphisms} We approximate $\phi \in \Homeo_0(\Sigma, \omega)$ by diffeomorphisms with the same rotation vector. Let $\psi \in \Diff_0(\Sigma, \omega)$ be any diffeomorphism such that $\mathcal{F}(\psi) = \mathcal{F}(\phi)$. It follows that $\mathcal{F}(\psi^{-1}\circ\phi) = 0$, so $\psi^{-1} \circ \phi$ lies in $\overline{\operatorname{Ham}}(\Sigma, \omega)$ \cite[Section $6$]{fathi}. Pick any sequence of Hamiltonian diffeomorphisms $h_k \in \operatorname{Ham}(\Sigma, \omega)$ approximating $\psi^{-1}\circ\phi$. The maps $\phi_k := \psi \circ h_k$ converge in the $C^0$ topology to $\phi$ and all have $\mathcal{F}(\phi_k) = \mathcal{F}(\phi)$. By the argument above, for any $q$ such that $p/q \in (0, c]$, each diffeomorphism $\phi_k$ has a non-contractible periodic point $z_k$ with minimal period $\leq q + g - 1$. Any subsequential limit $z$ is a non-contractible periodic point of $\phi$ with minimal period $\leq q + g - 1$.

\subsection{Proof of Proposition~\ref{prop:non_contractible_torus_technical}} \label{subsec:torus_proof} Assume that $\phi \in \Diff_0(\mathbb{T}^2, dx \wedge dy)$ has a fixed point $p \in \mathbb{T}^2$. If it is not $\Phi$-contractible, we are done. If it is $\Phi$-contractible, then we assume without loss of generality that it is fixed by $\Phi$. The same blow-up argument as in the genus $\geq 2$ case proves the proposition. 

\bibliographystyle{alpha}
\bibliography{main}

\end{document}